\documentclass[a4paper,11pt]{amsart}

\usepackage{amsmath}
\usepackage{amssymb}
\usepackage{amsthm}
\usepackage[dvipdfmx]{graphicx}
\usepackage{amscd}
\usepackage{comment}
\usepackage{color}
\usepackage{url}  

\newcommand{\supp}{\mathop{\mathrm{supp}}\nolimits}

\newcommand{\Isom}{\mathrm{Isom}}                         

\begin{document}

\numberwithin{equation}{section}

\theoremstyle{definition}
\newtheorem{theorem}{Theorem}[section]
\newtheorem{proposition}[theorem]{Proposition}
\newtheorem{formula}{Formula}[section]
\newtheorem{lemma}[theorem]{Lemma}
\newtheorem{corollary}[theorem]{Corollary}

\newtheorem{remark}{Remark}[section]
\newtheorem{definition}{Definition}[section]
\newtheorem{example}{Example}[section]
\newtheorem{problem}{Problem}[section]
\newtheorem{convention}{Convention}[section]
\newtheorem{conjecture}{Conjecture}[section]

\title[Actions of the Higman-Thompson groups on CAT(0) spaces]{Semi-simple actions of the Higman-Thompson groups $T_n$ on finite-dimensional CAT(0) spaces}
\author[M. Kato]{Motoko Kato}
\thanks{The author is supported by a JSPS KAKENHI Grant-in-Aid for Research Activity Start-up (grant number 19K23406), a JSPS KAKENHI Grant-in-Aid for Young Scientists (grant number 20K14311) and JST, ACT-X Grant Number JPMJAX200A, Japan.}
\address[M. Kato]{Faculty of Education, University of the Ryukyus, 1 Sembaru, Nishihara-Cho, Nakagami-Gun, Okinawa, 903-0213 Japan}
\email{katom@edu.u-ryukyu.ac.jp}      

\maketitle
\begin{abstract}
In this paper, we study isometric actions on finite-dimensional CAT(0) spaces for the Higman-Thompson groups $T_n$, which are generalizations of Thompson's group $T$.
It is known that every semi-simple action of $T$ on a complete CAT(0) space of finite covering dimension has a global fixed point.
After this result, we show that every semi-simple action of $T_n$ on a complete CAT(0) space of finite covering dimension has a global fixed point.
In the proof, we regard $T_n$ as ring groups of homeomorphisms of $S^1$ introduced by Kim, Koberda and Lodha, and use general facts on these groups.
\end{abstract}

\section{Introduction}
{\it Thompson's groups} $F$, $T$ and $V$ are finitely presented groups of homeomorphisms of the unit interval $[0,1]$, the circle and the Cantor space, respectively.
These groups were originally studied by Richard Thompson in the 1960s, and are known to have many interesting properties.
In particular, $T$ and $V$ were the first examples of finitely presented simple groups.
There are various generalizations of Thompson's groups, for example, {\it the Higman-Thompson groups} $F_n$, $T_n$ and $V_n$.
These groups are ``$n$-adic'' versions of $F$, $T$ and $V$, where $n$ is a natural number greater than $2$
($F_2$, $T_2$ and $V_2$ coincide with $F$, $T$ and $V$, respectively).

We treat actions of these groups on {\it CAT(0) spaces}:
geodesic spaces whose geodesic triangles are not fatter than the comparison triangles in the Euclidean plane.
We study the existence of fixed points of group actions on CAT(0) spaces.
Such properties, which we call fixed point properties, are related to many other properties of groups.
For example, a group is said to have {\it Serre's property FA} if every cellular action on every simplicial tree (that is, a 1-dimensional CAT(0) cube complex) has a global fixed point.
If a group has Serre's property FA, then the group does not split as an amalgamated product or an HNN extension (\cite{Serre}).
For every $k\geq 2$, a group is said to have {\it property $\mathrm{F}\mathcal{A}_k$} if every isometric action of the group on a complete CAT(0) space of covering dimension $k$ has a global fixed point (\cite{Farb} and \cite{Varghese}).
Property $\mathrm{F}\mathcal{A}_k$ is related to representations of groups: if a group has $\mathrm{F}\mathcal{A}_{k-1}$, then the group is of integral $k$-representation type [6].
Although it is generally hard to check whether a specific group has such fixed point properties, $T$ and $V$ are known to serve as examples of groups with these properties:
\begin{itemize}
\item $T$ and $V$ have Serre's property FA (\cite{Farley2}).
\item $T$, $V$, $T_n$ and $V_n$ act properly and isometrically on infinite dimensional CAT(0) cube complexes (\cite{GS}, \cite{Farley}).
\item For $T$, $V$, $V_n$ and some other generalizations of $V$, every semi-simple isometric action on a complete CAT(0) space of finite covering dimension has a global fixed point (\cite{Kato2}, see also \cite{Genevois} for the case of finite dimensional CAT(0) cube complexes). In other words, these groups have {\it property $\mathrm{F}\mathcal{A}_k$ for semi-simple actions} for every $k\in \mathbb{N}$. 
\end{itemize}
On the last result (\cite{Kato2}), we consider semi-simple isometric action on a complete CAT(0) space of finite covering dimension, in the case of $T_n$.

We regard each $T_n$ as a \textit{ring group}.
For $n\in \mathbb{Z}_{\geq 3}$, an $n$-ring group is generated by a sequence of $n$ orientation preserving homeomorphisms of $S^1$ whose supports form a ``ring'' of open intervals in $S^1$.
Further, it is requested that subgroups generated by two of the generators are isomorphic to $F$ when the supports have nonempty intersection.
Ring groups are $S^1$-versions of \textit{chain groups}.
For $n\in \mathbb{Z}_{\geq 2}$, an $n$-chain group is generated by a sequence of orientation preserving homeomorphisms of $\mathbb{R}$ whose supports form a ``chain'' of open intervals.
It is also requested that subgroups generated by two of the generators are isomorphic to $F$ when the supports have nonempty intersection.
In particular, every $2$-chain group is isomorphic to $F$.
By taking subsequences of generators, we may consider chain subgroups of chain and ring groups.
In particular, by considering $2$-chain subgroups, we can find many subgroups of chain and ring groups which are isomorphic to $F$.
Chain and ring groups are considered by Kim, Koberda and Lodha. In \cite{KKL}, they studied various general properties of chain groups.
To the author's knowledge, general properties of ring groups are not studied yet.

Our main theorems can be stated as follows.
\begin{proposition}\label{Tn_ring}
Let $n\geq 2$. There is a generating set of $T_n$ as an $(n^2+n)$-ring group with a minimal $n$-chain subgroup.
\end{proposition}

\begin{theorem}\label{TnFAk}
Let $m\geq 6$.
Let $G_{\mathcal{F}}$ be an $m$-ring group. 
Assume that $G_{\mathcal{F}}$ has a minimal $m'$-chain subgroup with $2\leq m'\leq m-4$. 
Let $H$ be a finitely generated subgroup of $[G_{\mathcal{F}}, G_{\mathcal{F}}]$.
Then for every semi-simple isometric action of $G_{\mathcal{F}}$ on a complete CAT(0) space of finite covering dimension, elements of $H$ have a common fixed point.
\end{theorem}

\begin{corollary}\label{Tn_cor}
For every $n\geq 2$, every semi-simple isometric action of $T_n$ on a complete CAT(0) space of finite covering dimension has a global fixed point.
In other words, $T_n$ has property $\mathrm{F}\mathcal{A}_k$ for semi-simple actions for every $k\in \mathbb{N}$. 
\end{corollary}

The outline of the proof of the main theorems are as follows.
In order to prove Proposition~\ref{Tn_ring}, we explicitly construct a generating set of $T_n$ as a ring group.
Corollary~\ref{Tn_cor} is a consequence of Proposition~\ref{Tn_ring} and Theorem~\ref{TnFAk}.
In order to prove Theorem~\ref{TnFAk}, we first observe that elements of the commutator subgroup of $F$ always admit fixed points (Lemma~\ref{F}).
Since $H$ is generated by conjugates of some copies of the commutator subgroup of $F$ in $G_{\mathcal{F}}$, we get a finite generating set of $H$ consisting of elements that always admit fixed points.
Next, we exchange $H$ with a larger subgroup with a finite generating set consisting of elements of the ``small'' supports.
To deal with a technical difficulty in this step, we use the assumption on the 
 chain subgroup (Lemma~\ref{minimal_lem}).
By the small support condition, we may apply general facts on semi-simple isometries of complete CAT(0) spaces of finite covering dimension (Theorem~\ref{Bridson} and Theorem~\ref{MainNew}), and show that these generators have a common fixed point.

Let me compare above proof with the proof in the case of $T$ in \cite{Kato2}, where we do not use the idea of ring groups.
In \cite[Theorem 1.1]{Kato2}, a sufficient condition for the fixed point property is given by a sequence of subgroups with no nontrivial homomorphisms to $\mathbb{R}$.
Such a sequence of subgroups is constructed with isometric copies of the subgroup of $F$ consisting of elements which are trivial in neighborhoods of $0$ and $1$ (\cite[Corollary 3.3]{Kato2}). 
Such a subgroup of $F$ coincides with the commutator subgroup of $F$, which is simple (\cite[Theorems 4.4 and 4.5]{CFP}).
It follows that the subgroup does not have nontrivial homomorphisms to $\mathbb{R}$.
If we try to literally generalize this proof in the case of $T_n$,
we consider the subgroup of $F_n$ consisting of elements which are trivial in neighborhoods of $0$ and $1$.
However, when $n\geq 3$, such a subgroup has a nontrivial homomorphism to $\mathbb{Z}^{n-2}$, and thus to $\mathbb{R}$ (\cite[Proof of Lemma 4.12]{Brown}, where $F_n$ is denoted by $F_{n,1}$ and the subgroup is denoted by $F_{n,1}^{0}$).
On the other hand, when we consider $T_n$ as a ring group as above, we may attribute the discussion on $T_n$ to the discussion on $2$-chain subgroups of $T$, which is isomorphic to $F$. 
Then we may avoid technical difficulties in $n$-adic case.

The rest of this paper is organized as follows.
In Section~\ref{chain_ring}, we recall precise definitions of $F$, $T$, $F_n$, $T_n$, 
chain groups and ring groups.
We introduce basic facts on chain and ring groups, mainly from \cite{KKL}.
We show Proposition~\ref{Tn_ring}. 
In Section~\ref{fixed_sec}, we discuss fixed point properties for Thompson's groups, chain and ring groups.
We start with the case of $F$ (Lemma~\ref{F}). We next discuss the case of chain groups (Thereom~\ref{FnFAk}). 
On these results, we show Theorem~\ref{TnFAk} and then Corollary~\ref{Tn_cor}.
The part of the proof that requires the assumption on the minimal chain subgroup is discussed independently as a lemma (Lemma~\ref{minimal_lem}).
We also treat the case of $F_n$ for independent interests (Corollary~\ref{fixed_Fn}).

\section{Thompson's groups, chain groups and ring groups}\label{chain_ring}

We gather some notations and conventions.
For a group $G$, we denote the commutator subgroup by $[G,G]$. 
We fix an orientation on $S^1$.
We define an {\it open interval} in $S^1$ as a subset of the form $(a,b)=\{t\in S^1\mid a< t< b\}$, where $a\neq b\in S^1$ and $<$ is the cyclic order on $S^1$. 
For an open interval $J$ in $S^1$, we write $a=\partial_{-}J$ and $b=\partial_{+}J$.
We also fix a basepoint on $S^1$, and identify $S^1$ with $[0,1]/\{0=1\}$ so that $0=1$ corresponds to the basepoint.
Through this identification, we identify subsets of $[0,1]$ with subsets of $S^1$. 
We also identify every homeomorphism of $[0,1]$ with the homeomorphism of $S^1$ that fix the basepoint.

Let $M$ be either $\mathbb{R}$ or $S^1$.
We write $\mathrm{Homeo}^+(M)$ for the group of orientation preserving homeomorphisms of $M$.
For $f\in \mathrm{Homeo}(M)$, we write $\supp(f)$ for the set-theoretic support of $f$:
$\supp(f)=\{t\in M\mid f(t)\neq t\}$.
For $S\subset \mathrm{Homeo}^{+}(M)$, we write $\supp(S)$ for the union of $\supp(s)$ for all $s\in S$.
We note that $\supp(f)$ and $\supp(S)$ are open subsets of $M$.

\subsection{Thompson's groups $F$, $T$ and the Higman-Thompson groups $F_n$, $T_n$}
In this section, we recall definitions and properties of the Higman-Thompson groups.
We start with definitions of {\it Thompson's groups} $F$ and $T$, following \cite{CFP}. 
$F$ is a group of piecewise affine homeomorphisms from the unit interval $[0, 1]$ to itself, differentiable
with derivatives of powers of $2$ on finitely many standard dyadic
intervals.
Here, a {\it standard dyadic interval}\/ is a subinterval of $[0,1]$ of the form $[a/{2^b}, {(a+1)}/{2^b}]$,
where $a$, $b$ are nonnegative integers. 
$T$ is a group of  piecewise affine homeomorphisms from the circle $S^1$ to itself that preserve the set of dyadic rationals and that are differentiable with derivatives
of powers of $2$ on finitely many standard dyadic intervals.
We note that $F$ is identified with the subgroup of $T$ consisting of elements that fix the basepoint of $S^1$.

Following \cite{CFP}, we represent each element of Thompson's groups by a pair of finite rooted binary trees with the same number of leaves.
When we represent elements of $T\setminus F$, we denote the image of the leftmost leaf of the domain tree by a small circle in the range tree.

As a generalization of $F$ and $T$,
we can consider $n$-adic versions of $F$ and $T$.
For every natural number $n\geq 2$, 
we define {\it standard $n$-adic intervals} as intervals of the form $[a/{n^b}, {(a+1)}/{n^b}]$ in $[0,1]$.
The {\it ($n$-adic) Higman-Thompson group} $F_{n}$ (resp.\ $T_n$) is a group of homeomorphisms of $[0, 1]$ (resp.\ $S^1$), differentiable with derivatives of powers of $n$ on
finitely many standard $n$-adic intervals (cf.\ Proposition 4.4 of \cite{Brown}). 
By definition, $F_2$ and $T_2$ coincide with $F$ and $T$.
We note that $F_n$ is identified with the subgroup of $T_n$ consisting of elements that fix the basepoint of $S^1$.

We identify $n$-adic divisions of $[0,1]$, that is, divisions of $[0,1]$ obtained by repeatedly dividing the interval into $n$ equal parts, with finite rooted $n$-ary trees.
With this identification, we represent each element of the $n$-adic Higman-Thompson group by a pair of finite rooted $n$-ary trees with the same number of leaves.
An {\it $n$-caret} in an $n$-ary tree, which is shown in Figure~\ref{n_caret}, corresponds to a division of an interval. 
Leaves of carets are ordered from left to right.
When we represent elements of $T_n\setminus F_n$, we denote the image of the leftmost leaf of the domain tree by a small circle in the range tree.

\begin{figure}[h]
\begin{center}
\scalebox{1}{
\includegraphics[width=5cm,pagebox=cropbox,clip]{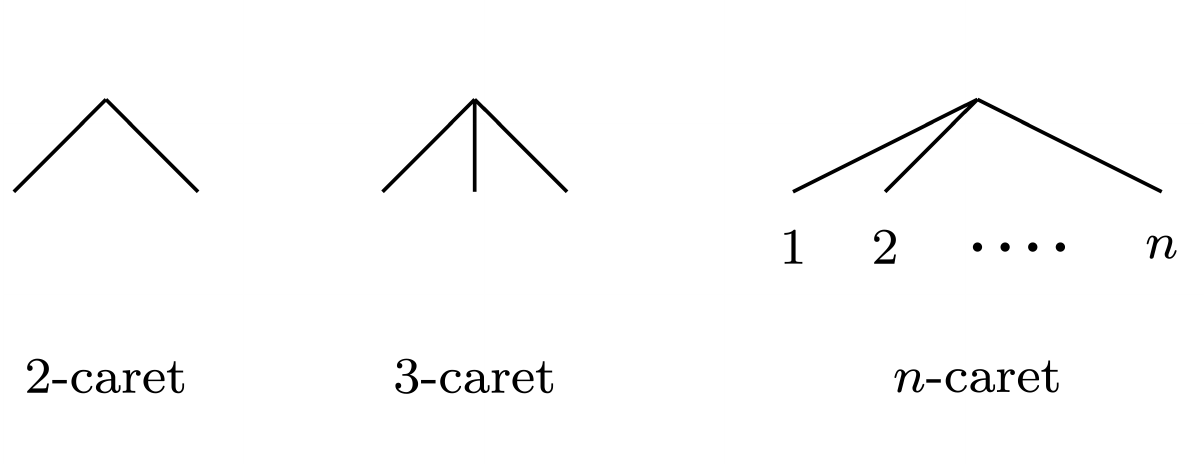}
}
\caption{$n$-carets}
\label{n_caret}
\end{center}
\end{figure}

For later use, we list some facts on $F_n$ and $T_n$.
\begin{lemma}\label{lemma-for-fact}
Let $n\geq 2$. 
Let $I=[a/n^b, (a+1)/{n^b}]$ be a standard $n$-adic interval in $[0,1]$.
If $f\in F_n$ is affine on $I=[a/n^b, (a+1)/{n^b}]$, then $f(I)$ is a standard $n$-adic interval of the form $[a'/n^{b'}, (a'+1)/n^{b'}]$, where $a\equiv a'$ mod $(n-1)$.
Conversely, if $I'=[a'/n^{b'}, (a'+1)/n^{b'}]$ is a standard $n$-adic interval in $[0,1]$ with $a\equiv a'$ mod $(n-1)$, then there exists $f\in F_n$ such that $f|_I$ is affine and that $f(I)=I'$.
\end{lemma}
For a proof, see Lemmas~1.1.3 and 1.2.1 of \cite{Brin-Guzman}, for example.

\begin{lemma}[{\cite[Section 4C]{Brown}}]\label{gen-F_n_lem}
For every natural number $n\geq 2$, $F_n$ is generated by $x_{n,i}$ $(1\leq i\leq n)$, where 
\begin{align}\label{x_i_eq}
x_{n,i}(t)=
  \begin{cases}
    t &(0\leq t\leq \frac{i-1}{n}) \\
    n(t-\frac{i-1}{n})+\frac{i-1}{n} &(\frac{i-1}{n}\leq t\leq \frac{i}{n}-\frac{i}{n^2})\\
    t-(\frac{i}{n}+\frac{i}{n^2})+\frac{n-1}{n} &(\frac{i}{n}-\frac{i}{n^2}\leq t\leq \frac{i}{n})\\
    \frac{1}{n}(t-1)+1 &(\frac{i}{n}\leq t\leq 1),
  \end{cases}
\end{align}
for $1\leq i\leq n-1$, and
\begin{align}\label{x_n_eq}
x_{n,n}(t)=
  \begin{cases}
    t &(0\leq t\leq 1-\frac{1}{n}) \\
    n(t-(1-\frac{1}{n}))+1-\frac{1}{n} &(1-\frac{1}{n}\leq t\leq 1-\frac{1}{n}+\frac{1}{n^2}-\frac{1}{n^3})\\
    t+\frac{1}{n}-\frac{2}{n^2}+\frac{1}{n^3}&(1-\frac{1}{n}+\frac{1}{n^2}-\frac{1}{n^3}\leq t\leq 1-\frac{1}{n}+\frac{1}{n^2})\\
    \frac{1}{n}(t-1)+1 &(1-\frac{1}{n}+\frac{1}{n^2}\leq t\leq 1).
  \end{cases}
\end{align}
$T_n$ is generated by $F_n$ and $y_{n}$, where
\begin{align}\label{y_n_eq}
y_{n}(t)=
  \begin{cases}
    t+1-\frac{1}{n} &(0\leq t\leq \frac{1}{n}) \\
    t-\frac{1}{n} &(\frac{1}{n}\leq t\leq 1).
  \end{cases}
\end{align}

\begin{figure}[h]
\begin{center}
\scalebox{2.5}{
\includegraphics[width=5cm,pagebox=cropbox,clip]{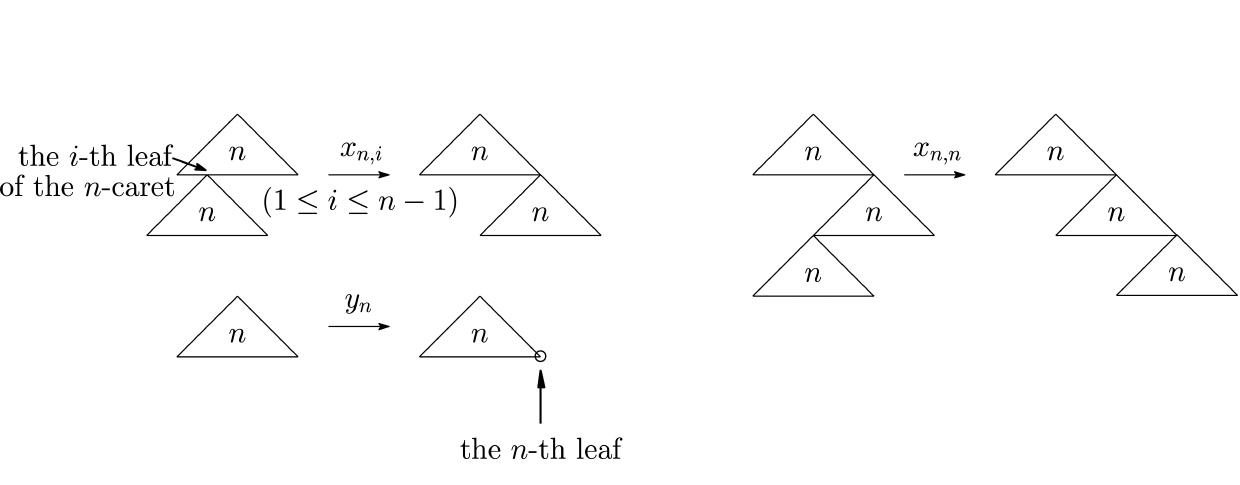}
}
\caption{Generators of $F_n$ and $T_n$}
\label{generators_of_T}
\end{center}
\end{figure}
\end{lemma} 
Figure~\ref{generators_of_T} shows tree pairs that represent generators of $F_n$ and $T_n$.

A useful property of the Higman-Thompson groups is their ``self-similarity'':
there exist subgroups which are isomorphic to the whole group, with arbitrarily small supports. 
Let $n\geq 2$.
For every standard $n$-adic interval $I$, let 
\begin{align}\label{subgroup_eq}
F_n(I)=\{f\in F_n\mid \supp(f)\subset I\}< F_n.
\end{align}
Note that every $F_n(I)$ is isomorphic to $F_n$ by means of a piecewise affine conjugation.
With a piecewise affine homeomorphism $\phi_{I}: I\to [0,1]$ such that all singularities are in $\mathbb{Z}[1/n]$ and that the derivative at any non-singular point is a power of $n$, 
we have an isomorphism $\psi_{n,I}:F_n(I)\to F_n$ defined by
\begin{align}\label{canonical-isom_eq}
\left(\psi_{n,I}(f)\right)(x)=({\phi_{I}}f{(\phi_{I})}^{-1})(x) \quad(x\in I).
\end{align}
When $I$ is obtained by an $n$-adic division of $[0,1]$, $\phi_{I}: I\to [0,1]$ is the affine homeomorphism from $I$ to $[0,1]$
(\cite[Lemma 4.4]{CFP} for the case of $n=2$).
We call $\psi_{n,I}$ a {\it canonical isomorphism} from $F_n(I)$ to $F_n$.

We often consider the natural action of $F_n$ (resp.\ $T_n$) on $[0,1]$ (resp.\ $S^1$) by homeomorphisms.
When we treat $F_n(I)$, we also consider the action of $F_n(I)$ on $I$ which is induced by the natural action of $F_n$ on $[0,1]$.
Let $G$ be a group of homeomorphisms of $\mathbb{R}$.
We say $G$ is {\it minimal} if every orbit is dense in $\supp(G)$.

\begin{lemma}\label{facts2}
Let $n\geq 2$. Let $I$ be a standard $n$-adic interval.
The natural action of $F_n(I)$ on $I$ is minimal.
\end{lemma}

\begin{proof}
We show that the natural action of $F_n$ on $[0,1]$ is minimal.
We take $x, y\in (0,1)$ and $\varepsilon>0$ arbitrarily. 
We fix standard $n$-adic intervals $I_x=[a/{n^b}, {(a+1)}/{n^b}]$ around $x$ and $I_y=[a'/{n^{b'}}, {(a'+1)}/{n^{b'}}]$ around $y$ in $(0,1)$.
Taking $b'\in \mathbb{N}$ sufficiently large, we may assume that $|1/{n^{b'}}|<\varepsilon$.
We further divide $I_y$ into $n$ subintervals $I_{y,i}=[{(na'+i)}/{n^{b'+1}}, {(na'+i+1)}/{n^{b'+1}}]$ ($0\leq i\leq n-1$).
We take $i\in \{0,\ldots, n-1\}$ such that $na'+i\equiv a$ mod $n$.
By Lemma~\ref{lemma-for-fact}, there exists $f\in F_n$ such that $f(I_x)=I_{y,i}$. Since $f(x)$ and $y$ are in $I_y$,
$|f(x)-y|<\varepsilon$.
It follows that the orbit $F_n\cdot x$ is dense in $[0,1]$.
It is clear that the action of $F_n(I)$ on $I$ is also minimal.
\end{proof}

\subsection{Chain groups and ring groups}

In this section, we treat another generalization of $F$, defined by Kim, Koberda and Lodha \cite{KKL}.
Let $m\geq 2$ be a natural number.
An {\it $m$-chain} is an ordered $m$-tuple of open intervals $(J_1,\ldots, J_m)$ in $\mathbb{R}$, satisfying that 
\begin{itemize}
\item $J_{i}\cap J_{i+1}$ is nonempty and connected for every $i\in \{1,\ldots, m-1\}$, and 
\item $J_i\cap J_j=\emptyset$ if $|i-j| \geq 2$.
\end{itemize}
Figure~\ref{chain} shows an example of a $4$-chain of intervals.
An {\it $m$-prechain group} $G_{\mathcal{F}}$ is a group 
generated by $\mathcal{F}=\{f_i\}_{1\leq i\leq n}\subset \mathrm{Homeo}^{+}(\mathbb{R})$ such that the sequence $(\supp(f_i))_{1\leq i\leq m}$ is an $m$-chain.
An {\it $m$-chain group} is an $m$-prechain group such that $\langle f_i, f_{i+1}\rangle$ is isomorphic to $F$ for every $1\leq i\leq m-1$. 
A group is a {\it chain group} if it is an $m$-chain group for some $m$.
A {\it chain subgroup} of a chain group $G_{\mathcal{F}}$ is a subgroup generated by elements of $\mathcal{F}$ 
with consecutive indices.
The family of chain groups is broad; in fact, there exist uncountably many isomorphism types of $3$-chain groups (Theorem 1.7 of \cite{KKL}).
It is known that $F_n$ is a ``stabilization'' of $n$-prechain groups. 
That is, for an $n$-prechain group $G_{\mathcal{F}}$ with $\mathcal{F}=\{f_i\}_{1\leq i\leq n}$, 
for all but finitely many $N\in \mathbb{N}$, the subgroup $\langle \{f_i^N\}_{1\leq i\leq n} \rangle$ of $G_{\mathcal{F}}$ is a chain group and
is isomorphic to $F_n$ (\cite[Proposition 1.10]{KKL}).

\begin{figure}
\begin{center}
\scalebox{1.5}{
\includegraphics[width=5cm,pagebox=cropbox,clip]{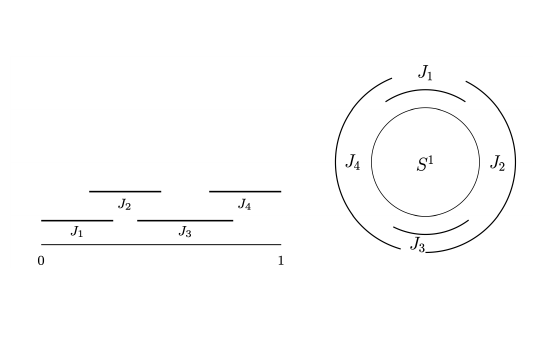}
}
\caption{A chain and a ring of intervals}
\label{chain}
\end{center}
\end{figure}

Let $m\geq 3$ be an integer.
An {\it $m$-ring} is an ordered $m$-tuple of open intervals $(J_1,\ldots, J_m)$ in $S^1$, satisfying that 
\begin{itemize}
\item $J_{i}\cap J_{i+1}$ is nonempty and connected for every $i$ modulo $m$, and 
\item $J_i\cap J_j=\emptyset$ if $2\leq |i-j| \leq m-2$.
\end{itemize}
An {\it $m$-ring group} $G_{\mathcal{F}}$ is a subgroup of $\mathrm{Homeo}^+(S^1)$
generated by $\mathcal{F}=\{f_i\}_{1\leq i\leq m}\subset \mathrm{Homeo}^+(S^1)$ such that 
\begin{itemize}
\item $(\supp(f_1),\ldots, \supp(f_m))$ is an $m$-ring, and 
\item $\langle f_i, f_{i+1}\rangle$ is isomorphic to $F$ for every $i$ modulo $m$.
\end{itemize} 
Figure~\ref{chain} shows an example of a $4$-ring of intervals.
A group is called a {\it ring group} if it is an $m$-ring group for some $m$.
{\it Chain subgroups} of an $m$-ring group are defined similarly as chain subgroups of chain groups, considering indices of generators as modulo $m$.

\begin{remark}
There is an alternative on the minimality of chain groups.
That is, for a chain group $G$, exactly one of the following holds (Theorem 1.3 of \cite{KKL}):
\begin{itemize}
\item The action of $G$ is minimal. In this case, $[G,G]$ is simple.
\item The closure of some $G$-orbit is perfect and totally disconnected. In this case,
$G$ canonically surjects onto an $n$-chain group that acts minimally.
\end{itemize}
\end{remark}

In the remainder of this section, we will introduce basic facts on chain and ring groups that will be used throughout subsequent sections.
\begin{lemma}[{\cite[Lemma 3.1]{KKL}}]\label{chainF}
Let $G_{\mathcal{F}}$ be a $2$-prechain group with respect to $\mathcal{F}=\{f_1, f_2\}$. 
Suppose that $f_2f_1(\partial_{-}\supp(f_2))\geq \partial_{+}\supp(f_1)$.
Then $\langle f_1, f_2\rangle\cong F$.
\end{lemma}

\begin{lemma}[{\cite[Lemma 3.6 (3), (4)]{KKL}}]\label{4.4}
Let $m\geq 2$.
Let $G_{\mathcal{F}}$ be an $m$-prechain group with respect to $\mathcal{F}=\{f_i\}_{1\leq i\leq m}$. 
\begin{itemize}
\item[(1)] For the natural action of $G_{\mathcal{F}}$ on $\supp(G_{\mathcal{F}})$, every $G_{\mathcal{F}}$-orbit is a $[G_{\mathcal{F}},G_{\mathcal{F}}]$-orbit.
\item[(2)] Let $x$ be a boundary point of $\supp(f)$ of some $f\in \mathcal{F}$.
For every proper compact subset $I\subset \supp(G_{\mathcal{F}})$ and every open neighborhood $J$ of $x$, there exists an element $g$ in $[G_{\mathcal{F}},G_{\mathcal{F}}]$
such that $g(I)\subset J$. In particular, the natural action of $[G_{\mathcal{F}},G_{\mathcal{F}}]$ on $\supp(G_{\mathcal{F}})$ is locally CO-transitive.
\item[(3)] If $G_{\mathcal{F}}$ is minimal, then for every proper compact subset $I$ and every nonempty open subset $J$ of $\supp(G_{\mathcal{F}})$, there exists an element $g$ in $[G_{\mathcal{F}},G_{\mathcal{F}}]$
such that $g(I)\subset J$. That is, the natural action of $[G_{\mathcal{F}},G_{\mathcal{F}}]$ on $\supp(G_{\mathcal{F}})$ is CO-transitive.
\item[(4)] For $m\geq 3$, if $G_{\mathcal{F}}$ has a minimal prechain subgroup, then $G_{\mathcal{F}}$ itself is minimal.
\item[(5)] For every $f\in [G_{\mathcal{F}}, G_{\mathcal{F}}]$, $\supp(f)$ is contained in a closed interval in $\supp(G_{\mathcal{F}})$.
\end{itemize}
\end{lemma}

\begin{proof}
(1) and (2) are Lemma 3.6 (3) and (4) of \cite{KKL}, respectively.

(3) By (1), if $G_{\mathcal{F}}$ is minimal, then $[G_{\mathcal{F}},G_{\mathcal{F}}]$ is also minimal.
Then (3) follows from (2) and Lemma 2.6 of \cite{KKL}, which states that a minimal locally CO-transitive group action is CO-transitive.

(4) Suppose that there exists a minimal prechain subgroup $G_{\mathcal{F'}}$ of $G_{\mathcal{F}}$. Let $\mathcal{F'}=\{f_{i}\}_{m_1\leq i\leq m_2}$. 
We show that $G_{\mathcal{F}}$ is minimal, that is, for every $x\in \supp(G_{\mathcal{F}})$, the $G_{\mathcal{F}}$-orbit of $x$ is dense in $\supp(G_{\mathcal{F}})$.

We fix $x\in \supp(G_{\mathcal{F}})$ arbitrarily. By (2), for every $y\in \supp(G_{\mathcal{F}})$ and $\varepsilon>0$, 
there exists $g_1\in G_{\mathcal{F}}$ such that 
$$g_1(x)\in \left(\partial_{+}\supp(f_{m_1})-\varepsilon, \partial_{+}\supp(f_{m_1})+\varepsilon\right).$$
In addition, there exists $g_2\in G_{\mathcal{F}}$ such that 
$$g_2([y-\varepsilon, y+\varepsilon])\subset \left(\partial_{+}\supp(f_{m_1})-\varepsilon, \partial_{+}\supp(f_{m_1})+\varepsilon\right).$$
We may assume that the $\varepsilon$-neighborhood of $\partial_{+}\supp(f_{m_1})$ is in $\supp\left(G_{\mathcal{F'}}\right)$.
Since $G_{\mathcal{F'}}$ is minimal, by (3), there exists $g_3\in G_{\mathcal{F'}}$ such that 
$$g_3(g_1(x))\in g_2\left((y-\varepsilon, y+\varepsilon)\right).$$
Then, $g_2^{-1}g_3g_1(x)\in (y-\varepsilon, y+\varepsilon)$. Hence the orbit of $x$ is dense in $\supp(G_{\mathcal{F}})$.

(5) Note that the commutator subgroup is normally generated by commutators of generators. 
Therefore, in the case of prechain groups, $[G_{\mathcal{F}}, G_{\mathcal{F}}]$ is normally generated by $[f_i,f_{i+1}]=f_i f_{i+1} f_{i}^{-1} f_{i+1}^{-1}$ $(1\leq i\leq m-1)$.
For every $i\in \{1,\ldots, m-1\}$, we can check that $\supp([f_i,f_{i+1}])$ is contained in a closed interval $[\partial_{-}\supp(f_{i+1}), f_2(\partial_{+}\supp(f_i))]$ in $\supp(G_{\mathcal{F}})$.
It follows that for every $f\in [G_{\mathcal{F}}, G_{\mathcal{F}}]$, $\supp(f)$ is contained in a closed interval $\supp(G_{\mathcal{F}})$.
\end{proof}

\begin{lemma}\label{5.7}
Let $m\geq 4$.
Let $G_{\mathcal{F}}$ be an $m$-ring group with respect to $\mathcal{F}=\{f_i\}_{1\leq i\leq m}$.
Then, for all $m'\geq m$, $G_{\mathcal{F}}$ can be described as a $m'$-ring group.
That is, there exists a generating set $\mathcal{F'}=\{f'_{j}\}_{1\leq j\leq m'}$ of $G_{\mathcal{F}}$ which satisfy the assumptions in the definition of the ring group.
Moreover, when $m\geq 5$, we can take such $\mathcal{F}'$ so that $f'_{j}=f_{j}$ for $(m-4)$ consecutive indices $j$.
\end{lemma}

\begin{proof}
Although the proof is similar to the corresponding result for chain groups (Proposition 1.5 of \cite{KKL}),
we add some arguments needed for ring groups.
It is enough to show the case where $m'=m+1$.

First, we give a new generating set for $G_{\mathcal{F}}$ as an $(m+1)$-ring group.
Let $\mathcal{F}=\{f_1,\ldots, f_m\}$.
We exchange generators with conjugates of the original ones.
Let 
\begin{align}
\begin{aligned}
&f'_{m-2}=f_{m-1}^{-N}f_{m-2}f_{m-1}^{N},\ \ f'_{m}=(f_{m-2}f_m)^N f_{m-1}^N(f_{m-2}f_m)^{-N},\\
&f'_{m-1}=(f'_m)^{-1}f_{m-1}f'_m,\ \ f'_{m+1}=f_m
\end{aligned} 
\end{align}
for some $N\in \mathbb{N}$.
We further define $f'_{i}$ $(1\leq i\leq m-3)$ by
\begin{align}
f'_j=
  \begin{cases}
  f_m^N f_1 f_m^{-N} &(j=1) \\
   f_j &(\text{otherwise}),
  \end{cases}
\end{align}
and let $\mathcal{F'}=\{f'_{j}\}_{1\leq j\leq m+1}$.
By the construction of $\mathcal{F}'$, $G_{\mathcal{F'}}=G_{\mathcal{F}}$ as a subgroup of $\mathrm{Homeo}^{+}(S^1)$.
If $N$ is sufficiently large, $(\supp(f'_{j}))_{2\leq j\leq m+1}$ is a chain of intervals and $\langle f'_j, f'_{j+1}\rangle=F$ for $2\leq j\leq m$ (\cite[Theorem 4.7]{KKL}).

Next, we confirm that 
$(\supp(f'_{j}))_{1\leq j\leq m+1}$ is a chain of intervals.
According to the proof of \cite[Theorem 4.7]{KKL}, $(\supp(f'_{j}))_{2\leq j\leq m+1}$ is a chain of intervals.
We confirm that $(\supp(f'_{j}))_{j=m, m+1, 1, 2}$ is a chain of intervals. 
In the following, we write $<$ for the induced order on some open interval in $S^1$.
Since $f'_{m}=(f_{m-1})^{(f_{m-2}f_{m})^N}$ and $f'_1=(f_{1})^{(f_{m})^N}$,
\begin{align*}
&\partial_{+}\supp(f'_{m})=(f_{m-2}f_{m})^N(\partial_{+}\supp(f_{m-1}^N))=(f_{m-2}f_{m})^N(\partial_{+}\supp(f_{m-1}))\\
&=f_{m}^N(\partial_{+}\supp(f_{m-1}))< f_{m}^N(\partial_{-}\supp(f_{1}))=\partial_{-}\supp(f'_1).
\end{align*}
Here, the third equality follows from an observation that $f_{m}^n(\partial_{+}\supp(f_{m-1}))\in \supp(f_{m})$ for all $n\in \mathbb{Z}$ and $f_{m-2}(x)=x$ for all $x\in \supp(f_{m})$.
Moreover, since $f_{m}$ fixes $\supp(f_{m})$, both $\partial_{+}\supp(f'_{m})$ and $\partial_{-}\supp(f'_{1})$ are in $\supp(f_{m})=\supp(f'_{m+1})$.
Therefore,
\begin{align*}
\partial_{-}\supp(f'_{m+1})<\partial_{+}\supp(f'_{m})<\partial_{-}\supp(f'_1)<\partial_{+}\supp(f'_{m+1}).
\end{align*}
Since $f'_2$ is either $f_2$ ($m\geq 5$) or $f_3^{-N}f_2f_3^{N}$ ($m=4$), we have $\partial_{-}\supp(f'_2)=\partial_{-}\supp(f_2)$ in either case. 
Therefore,
\begin{align*}
\partial_{+}\supp(f'_{m+1})<\partial_{-}\supp(f'_{2})<\partial_{+}\supp(f'_1).
\end{align*}
It follows that $(\supp(f'_{j}))_{j=m, m+1, 1, 2}$ is a chain of intervals.

Finally, we show that $\langle f'_j, f'_{j+1}\rangle=F$ for $1\leq j\leq m+1$, where indices are considered modulo $(m+1)$.
With sufficiently large $N$, $\langle f'_j, f'_{j+1}\rangle=F$ for $2\leq j\leq m$ (\cite[Theorem 4.7]{KKL}).
The remaining cases are $\langle f'_j, f'_{j+1}\rangle$ for 
$j=m+1$ and $1$.
For $j=m+1$, 
\begin{align*}
\langle f'_{m+1}, f'_{1}\rangle=\langle f_m, {{{f_{m}}^{N}}f_{1}}{f_{m}}^{-N}\rangle={{f_{m}}^{N}}\langle {f_m, f_{1}}\rangle{f_{m}}^{-N}\cong\langle f_{m}, f_1\rangle\cong F. 
\end{align*}
For $j=1$, when $m\geq 5$, we have $f'_2=f_2$ and thus
\begin{align*}
\langle f'_{1}, f'_{2}\rangle
&=\langle {{f_{m}}^{N}}{f_{1}}{{f_{m}}^{-N}}, f_2\rangle = \langle {{f_{m}}^{N}}{f_{1}}{{f_{m}}^{-N}},  {{f_{m}}^{N}}{f_{2}}{{f_{m}}^{-N}}\rangle \\
&= {{f_{m}}^{N}}\langle{f_{1}},  f_2\rangle{{f_{m}}^{-N}}\cong \langle f_1, f_2\rangle\cong F.
\end{align*}
Here, since $\supp(f'_2)=\supp(f_2)$ and $\supp(f_m)$ are disjoint, ${{f_{m}}^{N}}{f'_{2}}{{f_{m}}^{-N}}=f'_2$.
Similarly, when $m=4$, we have $f'_2=f_3^{-N}f_2f_3^{N}$ and thus
\begin{align*}
\langle f'_{1}, f'_{2}\rangle
&=\langle {{{f_{4}}^{N}}f_{1}{{f_{4}}^{-N}}}, f'_2\rangle=\langle {{f_{4}}^{N}}f_{1}{{f_{4}}^{-N}},  {{f_{4}}^{N}}f'_{2}{{f_{4}}^{-N}}\rangle \\
&\cong \langle f_1, f'_2\rangle=\langle f_1, {f_3^{-N}}f_2{f_3^{N}}\rangle=\langle {f_3^{-N}}f_1{f_3^{N}}, {f_3^{-N}}f_2{f_3^{N}}\rangle\\
&\cong \langle f_1, f_2\rangle \cong F.
\end{align*}
Here, since $\supp(f'_2)$ and $\supp(f_4)$ are disjoint, $f'_2={{{f_{4}}^{-N}}f'_2{{f_{4}}^{N}}}$.
Since $\supp(f_1)$ and $\supp(f_3)$ are disjoint, $f_1={{{f_{3}}^{-N}}f_1{{f_{3}}^{-N}}}$.

Therefore, $G_{\mathcal{F'}}$ is an $(m+1)$-ring group which is isomorphic to $G_{\mathcal{F}}$.
Note that $f'_{j}=f_{j}$ for $2\leq j\leq m-3$, and that the generators can be renumbered cyclically.
\end{proof}

\subsection{The Higman-Thompson groups $T_n$ are ring groups}
In this section, we show Proposition~\ref{Tn_ring}.
First, we construct generating sets for $F_n$ as a chain group. 
In Lemma~\ref{F_n-chain_lem}, we give two generating sets for each $F_n$.
The first generating set, denoted by $\{f_{i}\}_{1\leq i\leq n}$, can be constructed immediately from the standard generating set in Lemma~\ref{gen-F_n_lem}.
The second generating set, $\{g_j\}_{1\leq j\leq n^2+n-1}$, is a modification of the first one. 
With the second generating set, some subgroups of $F_n$ of the form $F_n(I)$ appears as chain subgroups of $F_n$.
We note that $F_n$ is minimal (Lemma~\ref{facts2}).
Finally, by extending the second generating set, we regard $T_n$ as a ring group with a minimal chain subgroup.
The existence of a minimal chain subgroup will be used in the proof of Corollary~\ref{Tn_cor}, which will be treated in Section~\ref{fixed_sec}.

\begin{lemma}\label{F_n-chain_lem}
Let $n\geq 2$.
\begin{itemize}
\item[(1)] There exists a generating set $\{f_{i}\}_{1\leq i\leq n}$ of $F_n$ such that
	\begin{itemize}
	\item[\rm{(1-i)}] $(\supp(f_{1}),\ldots, \supp(f_{n}))$ is an $n$-chain, and 
	\item[\rm{(1-ii)}] $\langle f_i, f_{i+1}\rangle$ is isomorphic to $F$ for every $i\in \{1,\ldots, n-1\}$.
	\end{itemize}
In particular, every $F_n$ is an $n$-chain group.
\item[(2)] There exists a generating set $\{g_j\}_{1\leq j \leq n^2+n-1}$ such that 
	\begin{itemize}
	\item[\rm{(2-i)}] $(\supp(g_{1}),\ldots, \supp(g_{n^2+n-1}))$ is an $(n^2+n-1)$-chain, 
	\item[\rm{(2-ii)}] $\langle g_j, g_{j+1}\rangle$ is isomorphic to $F$ for every $j\in \{1,\ldots, n^2+(n-1)-1\}$, and
	\item[\rm{(2-iii)}] for every $k\in \{1,\ldots, n\}$, a subgroup $F_n([(k-1)/n, k/n])$ defined in $(\ref{subgroup_eq})$ coincides with the subgroup generated by $\{g_j\}_{1+(k-1)(n+1)\leq j \leq n+(k-1)(n+1)}$.
	\end{itemize} 
In particular, all $F_n([(k-1)/n, k/n])$ can be described as chain subgroups of $F_n$ simultaneously.
\end{itemize}
\end{lemma}

\begin{proof}
(1) First, we define $\{f_{i}\}_{1\leq i\leq n}$ and show that it is a generating set of $F_n$.
Let
\begin{align}
f_{i}&=x_{n, i+1}^{-1}x_{n,i} \quad(1\leq i\leq n-1), \label{f_i_eq}\\ 
f_{n}&=x_{n,n},\label{f_n_eq}
\end{align}
where $\{x_{n,i}\}_{1\leq i\leq n}$ is a generating set of $F_n$ defined by $(\ref{x_i_eq})$ and $(\ref{x_n_eq})$ in Lemma~\ref{gen-F_n_lem} (see also Figure~\ref{generators_of_T}).
Figure~\ref{newgen} shows tree pairs that represent $f_{i}$ $(1\leq i\leq n)$. 
It follows that
$$x_{n,i}=f_{n}f_{n-1}\cdots f_{i} \quad(1\leq i\leq n),$$ and thus $\{f_{i}\}_{1\leq i\leq n}$ is another generating set of $F_n$.
Figure~\ref{newgen} shows tree pairs that represent $f_{i}$ $(1\leq i\leq n)$. 

Next, we confirm (1-i).
By $(\ref{f_i_eq})$ and $(\ref{f_n_eq})$,
\begin{align*}
\begin{aligned}
\displaystyle &\supp(f_{i})=\left(\frac{i-1}{n}, \frac{i+1}{n}\right)\quad (1\leq i\leq n-2),\\ 
&\supp(f_{n-1})=\left(1-\frac{2}{n},1-\frac{1}{n}+\frac{1}{n^2}\right),\ \ 
\supp(f_{n})=\left(1-\frac{1}{n},1\right).
\end{aligned}
\end{align*}
Therefore, $\supp(f_{i})$ $(1\leq i\leq n)$ form a chain of open intervals.

Finally, we confirm (1-ii).
We use Lemma~\ref{chainF}.
We have
\begin{align*}
\begin{aligned}
\displaystyle &f_{i+1}f_{i}\left(\frac{i+1}{n}\right)=\frac{i+2}{n}-\frac{2}{n^2}\geq \frac{i+1}{n} \quad(1\leq i\leq n-3),\\
&f_{n-1}f_{n-2}\left(1-\frac{2}{n}\right)
=1-\frac{1}{n}+\frac{1}{n^2}-\frac{1}{n^3}\geq 1-\frac{1}{n},\\
&f_{n}f_{n-1}\left(1-\frac{1}{n}\right)
\geq 1-\frac{1}{n^2}\geq 1-\frac{1}{n}+\frac{1}{n^2}.
\end{aligned}
\end{align*}
By Lemma~\ref{chainF}, $\langle f_{i}, f_{i+1}\rangle$ $(1\leq i\leq n-1)$ are isomorphic to $F$.
It follows that $F_n$ is an $n$-chain group for every $n\geq 2$.

\begin{figure}
\begin{center}
\scalebox{2.5}{
\includegraphics[width=5cm,pagebox=cropbox,clip]{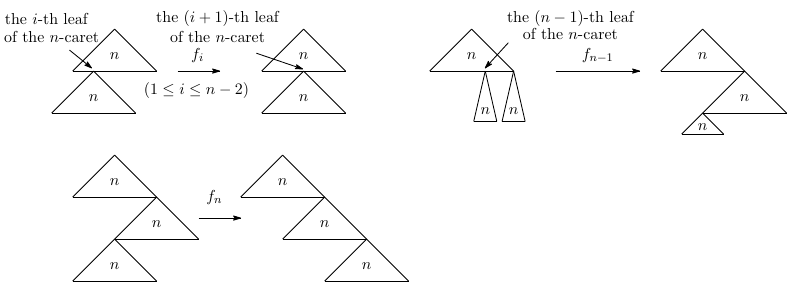}
}
\caption{Generators $\{f_{i}\}_{1\leq i\leq n}$ of $F_n$}
\label{newgen}
\end{center}
\end{figure}

(2) First, we define $\{g_{j}\}_{1\leq j\leq n^2+(n-1)}$ and show that it is a generating set of $F_n$.
Recall that for a standard $n$-adic interval $I$ obtained by an $n$-adic division of $[0,1]$, $\phi_{I}: I\to [0,1]$ is the affine map and that 
$\psi_{n,I}:F_n(I)\to F_n$ is the canonical isomorphism 
(see $(\ref{canonical-isom_eq})$).
For every $k\in \{1,\ldots, n\}$, let
\begin{align}\label{g_ki_eq}
g_{k,i}&=\psi_{n,[(k-1)/n, k/n]}\left(f_i\right) \quad(1\leq i\leq n),
\end{align}
where $\{f_i\}_{1\leq i\leq n}$ is the generating set of $F_n$ defined in $(1)$.
We further define elements $\{h_{k}\}_{1\leq k\leq n-1}$ in $F_n$ by
\begin{align}\label{h_k-eq}
h_k=\left(\psi_{n,[k/n, (k+1)/n]}\left(x_{n,1}^{-1}\right)\right) \circ f_k \circ \left(\psi_{n,[(k-1)/n, k/n]}\left(x_{n,1}^{-1}\right)\right).
\end{align}
\begin{figure}
\begin{center}
\scalebox{2}{
\includegraphics[width=5cm,pagebox=cropbox,clip]{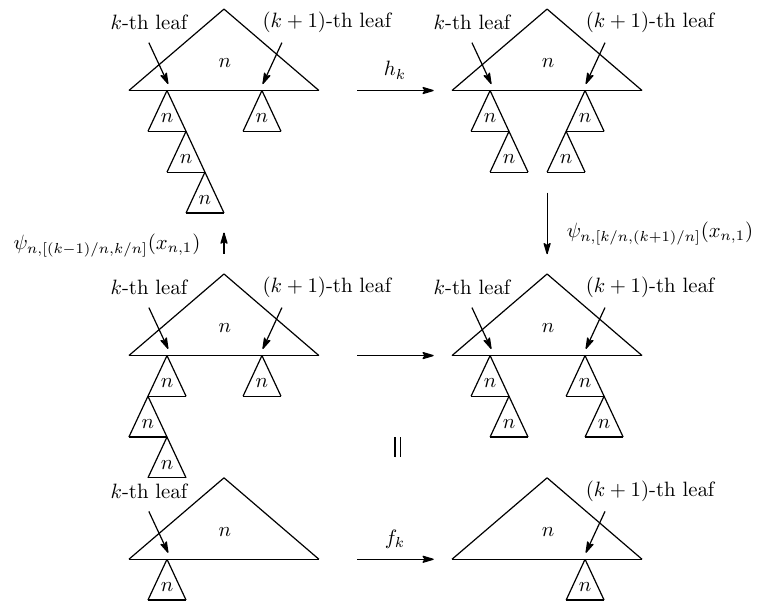}
}
\caption{The construction of $h_k$}
\label{hk}
\end{center}
\end{figure}
Figure~\ref{hk} shows tree pairs that represent $h_k$.
Using these elements, we define a sequence of elements $(g_1, g_2, \ldots, g_{n^2+(n-1)})$ by
\begin{align}\label{gen-eq}
(\underbrace{g_{1,1},\ldots,g_{1,n}}_{\{g_{1,i}\}_{1\leq i\leq n}},h_1, \ldots, h_{k-1}, \underbrace{g_{k,1},\ldots, g_{k,n}}_{\{g_{k,i}\}_{1\leq i\leq n}}, h_{k},\ldots, h_{n-1}, \underbrace{g_{n,1},\ldots, g_{n,n}}_{\{g_{n,i}\}_{1\leq i\leq n}}).
\end{align}
It follows that
\begin{align*}
\psi_{n,[k/n, (k+1)/n]}\left(x_{n,1}^{-1}\right)\in \langle \{g_{k+1,i}\}_{1\leq i\leq n} \rangle.
\end{align*}
Similarly, 
\begin{align*}
\psi_{n,[(k-1)/n, k/n]}\left(x_{n,1}^{-1}\right)\in \langle \{g_{k,i}\}_{1\leq i\leq n} \rangle.
\end{align*}
Then, by (\ref{h_k-eq}), 
$$f_k\in \langle \{g_{k,i}\}_{1\leq i\leq n}, h_k \rangle$$
for every $k\in \{1,\ldots, n-1\}$.
In addition,
$$f_n=x_{n,n}\in F_n([n-1/n, 1])=\langle \{g_{n,i}\}_{1\leq i\leq n} \rangle.$$
Therefore, $\{g_{j}\}_{1\leq j\leq n^2+(n-1)}$ generates $F_n$.

Next, we confirm (2-i).
By $(\ref{canonical-isom_eq})$ and $(\ref{g_ki_eq})$,
$$\supp(g_{k,i})=(\phi_{[(k-1)/n, k/n]})^{-1}\left(\supp(f_i)\right)$$
for every $k$ and $i$. 
In other words, intervals $\left(\supp(g_{k,i})\right)_{1\leq i\leq k}$ on $[(k-1)/n, k/n]$ are affine copies of the intervals $\left(\supp(f_{i})\right)_{1\leq i\leq k}$ on $[0,1]$.
By $(1)$, $\left(\supp(f_{i})\right)_{1\leq i\leq k}$ is a chain that covers $(0,1)$.
It follows that $\left(\supp(g_{k,i})\right)_{1\leq i\leq k}$ is a chain that covers $(k/n, (k+1)/n)$.
Moreover, we have
$$\supp(h_k)=\left(\frac{k}{n}-\frac{1}{n^3}, \frac{k}{n}+\frac{1}{n^2}\right).$$
Note that
$$\partial_{-}\supp(h_k)<\partial_{+}\supp(g_{k,n})=\partial_{-}\supp(g_{k+1,1})=\frac{k}{n}.$$
Note also that
$$\partial_{-}\supp(h_k)\geq \partial_{+}\supp(g_{k,n-1})=\frac{k}{n}-\frac{1}{n^2}+\frac{1}{n^3},$$
and
$$\partial_{+}\supp(h_k)= \partial_{-}\supp(g_{k+1,2}).$$
Then $(\supp(g_j))_{1\leq j\leq n^2+(n-1)}$ is an $(n^2+(n-1))$-chain.

Next, we confirm (2-ii).
Through isomorphisms $\psi_{n,[(k-1)/n, k/n]}$, (2-ii) corresponds to (1-ii), except for the cases of $\langle g_{k-1,n}, h_k\rangle$ and $\langle h_{k}, g_{k+1,2}\rangle$ ($k\in \{1,\ldots, n-1\}$).
Then it is sufficient to show that $\langle g_{k-1,n}, h_k\rangle$ and $\langle h_{k}, g_{k+1,2}\rangle$ are isomorphic to $F$ for every $k\in \{1,\ldots, n-1\}$.
We have
\begin{align*}
&h_kg_{k-1,n}\left(\partial_{-}\supp(h_k)\right)=h_kg_{k-1,n}\left(\frac{k}{n}-\frac{1}{n^3}\right)=\frac{k}{n}+\frac{2}{n^2}-\frac{1}{n^3}\\
                                                                         &\geq \frac{k}{n}=\partial_{+}\supp(g_{k-1,n}).
\end{align*}
By Lemma~\ref{chainF}, $\langle g_{k-1,n}, h_k\rangle\simeq F$. Similarly,
\begin{align*}
&g_{k+1,2}h_k\left(\partial_{-}\supp(g_{k+1,2})\right)=g_{k+1,2}h_k\left(\frac{k}{n}\right)=\frac{k}{n}+\frac{2}{n^2}-\frac{2}{n^3}\\
&> \frac{k}{n}+\frac{1}{n^2}-\frac{1}{n^3}=\partial_{+}\supp(h_{k}).
\end{align*}
Then $\langle h_{k}, g_{k+1,2}\rangle\simeq F$.

Finally, we confirm (2-iii). This is immediate from the definition of $g_j$.
In fact, for every $k\in \{1,\ldots,n\}$, $\{g_{k,i}\}_{1\leq i\leq n}$ generates $F_n([(k-1)/n, k/n])$ since $\psi_{n,[(k-1)/n, k/n]}$ is a group isomorphism.
Together with (2-i) and (2-ii), when we regard $F_n$ as a chain group with the generating set $\{g_{j}\}_{1\leq j\leq n^2+(n-1)}$, each $F_n([(k-1)/n, k/n])$ is a chain subgroup of $F_n$.
\end{proof}

\begin{proof}[Proof of Proposition~\ref{Tn_ring}]
Let $n\geq 2$.
It is sufficient to construct a generating set $\{t_j\}_{1\leq j \leq n^2+n}$ of $T_n$ such that 
	\begin{itemize}
	\item[(i)] $(\supp(t_{1}),\ldots, \supp(t_{n^2+n}))$ is an $(n^2+n)$-ring, 
	\item[(ii)] $\langle t_j, t_{j+1}\rangle$ is isomorphic to $F$ for every $j$ mod $(n^2+n)$,
	\item[(iii)] $F_n$ coincides with the subgroup generated by $\{t_j\}_{1\leq j \leq n^2+n-1}$, and
	\item[(iv)] $F_n([(k-1)/n, k/n])$ coincides with the subgroup generated by 
	$$\{t_j\}_{1+(k-1)(n+1)\leq j \leq n+(k-1)(n+1)}$$
	 for every $k\in \{1,\ldots, n\}$.
	\end{itemize} 
In fact, with such a generating set, each $T_n$ is regarded as an $(n^2+n)$-ring group with an $n$-chain subgroup $F_n([1-\frac{1}{n}, 1])$, which is minimal (Lemma~\ref{facts2}). 

First, we define $\{t_j\}_{1\leq j \leq n^2+n}\subset T_n$ by adding one element to $\{g_{j}\}_{1\leq j\leq n^2+n-1}$ in Lemma~\ref{F_n-chain_lem} (2).
Let
\begin{align}\label{t_j_eq}
t_j=
\begin{cases}
g_j &(1\leq j\leq n^2+n-1),\\
{y_n}g_{n+1}y_n^{-1}&(j=n^2+n),
\end{cases}
\end{align}
where $y_{n}$ is defined by $(\ref{y_n_eq})$ in Lemma~\ref{gen-F_n_lem} (see also Figure~\ref{generators_of_T}).
Figure~\ref{fn+1} shows tree pairs that represent $t_{n^2+n}$. 
We denote the image of the leftmost leaf of the domain tree by a small circle in the range tree.
\begin{figure}
\begin{center}
\scalebox{1.5}{
\includegraphics[width=5cm,pagebox=cropbox,clip]{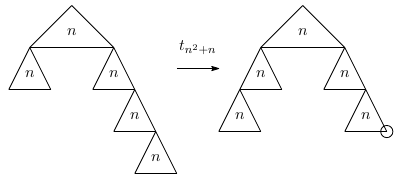}
}
\caption{$t_{n^2+n}$ in $T_n$}
\label{fn+1}
\end{center}
\end{figure}

Next, we show that $\{t_{i}\}_{1\leq i\leq n^2+n}$ generates $T_n$.
Recall that $T_n$ is generated by $F_n\cup \{y_n\}$ in Lemma~\ref{gen-F_n_lem}.
It is sufficient to show that $y_n\in \langle F_n\cup \{t_{n^2+n}\}\rangle$.
Let $I=[a,b]$ be a standard $n$-adic interval such that $y_n|_{I}$ is affine and that $y_n(I)\subset (0,1)$.
By Lemma~\ref{lemma-for-fact}, we take $f\in F_n$ such that $f|_{y_n(I)}$ is affine and that $f(y_n(I))=I$.
Then $\supp(f y_n)\subset S^1\setminus I$.
We take $f'\in F_n$ such that $f'(I)\supset [1/{n^2},1-1/{n^2}]$ (see Lemmas~\ref{facts2} and \ref{4.4} (3)).
Then $\supp({f'}(fy_n)f'^{-1})\subset S^1\setminus [1/{n^2},1-1/{n^2}]$.
Since 
$$t_{n^2+n}([1-1/{n^2},1-1/{n^2}+1/{n^3}])=S^1\setminus [1/{n^2},1-1/{n^2}],$$
it follows that 
$$\supp(t_{n^2+n}^{-1}f'(fy_n){(t_{n^2+n}^{-1}f')}^{-1})\subset [1-1/{n^2},1-1/{n^2}+1/{n^3}].$$ 
Let $f''=t_{n^2+n}^{-1}f'(fy_n){(t_{n^2+n}^{-1}f')}^{-1}$.
We have $y_n=f^{-1}{f'^{-1}t_{n^2+n}}f''t_{n^2+n}^{-1}f'$, where $f,f',f''\in F_n$.
Therefore, $y_n\in \langle F_n\cup \{t_{n^2+n}\}\rangle$.

Finally, we confirm that $\{t_{j}\}_{1\leq j\leq n^2+n}$ satisfies four requirements.
Since (iii) and (iv) are related only to $t_j$ of $1\leq j\leq n^2+n-1$, which corresponds to $g_j$, (iii) and (iv) are satisfied by Lemma~\ref{F_n-chain_lem} (2).
We confirm (i) and (ii).
By definition of $t_j$, (i) and (ii) correspond to (2-i) and (2-ii) in Lemma~\ref{F_n-chain_lem}, except as related to $t_{n^2+n}$.
Namely, it is sufficient to confirm that
	\begin{itemize}
	\item[(i)'] $(t_{n^2+n-2}, t_{n^2+n-1}, t_{n^2+n}, t_1, t_2)$ is a $5$-ring, and
	\item[(ii)'] $\langle t_{n^2+n-1}, t_{n^2+n}\rangle$ and  $\langle t_{n^2+n}, t_{1}\rangle$ are isomorphic to $F$.
	\end{itemize} 
Since $$y_n|_{[0, 1/n]}={(\phi_{[1-1/n, 1]})}^{-1}\phi_{[0, 1/n]},$$ for $j\in \{1,\ldots,n\}$,
\begin{align*}
&{y_n}t_{j}y_n^{-1}=\left({{(\phi_{[1-1/n, 1]})}^{-1}}{\phi_{[0, 1/n]}}\right)t_j{\left({{(\phi_{[1-1/n, 1]})}^{-1}}{\phi_{[0, 1/n]}}\right)}^{-1}\\
&=\left({{(\phi_{[1-1/n, 1]})}^{-1}}{\phi_{[0, 1/n]}}\right)g_j{\left({{(\phi_{[1-1/n, 1]})}^{-1}}{\phi_{[0, 1/n]}}\right)}^{-1}\\
&={\phi_{[1-1/n, 1]}}^{-1}\left(\phi_{[0, 1/n]}g_j{\phi_{[0, 1/n]}}^{-1}\right){\phi_{[1-1/n, 1]}}\\
&={{(\phi_{[1-1/n, 1]})}^{-1}}f_j{\phi_{[1-1/n, 1]}}=t_{j-(n+1)+(n^2+n)}.
\end{align*}
Similarly, ${y_n}t_{n+2}{y_n}^{-1}=t_1$ and ${y_n}t_{n+3}{y_n}^{-1}=t_2$.
By definition, ${y_n}t_{n+1}{y_n}^{-1}=t_{n^2+n}$.
Therefore,
$$t_{n^2+n-2}, t_{n^2+n-1}, t_{n^2+n}, t_1, t_2$$ are the conjugates of $$t_{n-1}, t_{n}, t_{n+1}, t_{n+2}, t_{n+3}$$ by $y_n$, respectively in $T_n$.
Then (i)' and (ii)' come down to when $t_{n^2+n}$ is not relevant.
\end{proof}

\section{Fixed point properties}\label{fixed_sec}
In this section, we show the main results: Theorem~\ref{TnFAk} and Corollary~\ref{Tn_cor}.
First, we cite general facts on group actions on finite-dimensional CAT(0) spaces: Theorems~\ref{Bridson} ({\cite[Proposition 3.4]{Bridson}}) and \ref{MainNew} ({\cite[Theorem 1.1]{Kato2}}).
These theorems are used in the proof of the following lemmas and theorems.
Next, we discuss fixed point properties in the case of $F$ (Lemma~\ref{F}).
Using Lemma~\ref{F}, we discuss the case of of chain groups (Thereom~\ref{FnFAk}).
We add Corollary~\ref{fixed_Fn} for independent interests.
Using Lemma~\ref{F} and Theorem~\ref{FnFAk}, we discuss the case of ring groups (Theorem~\ref{TnFAk}).
The part of the proof that requires the assumption on the minimal chain subgroup is discussed independently as a lemma (Lemma~\ref{minimal_lem}).
Finally, we discuss the case of $T_n$ (Corollary~\ref{Tn_cor}), which is a consequence of Theorem~\ref{TnFAk} and Proposition~\ref{Tn_ring}.

{\it The Lebesgue covering dimension} of a topological space $X$ is the minimum $k\in \mathbb{N}$ such that every finite open cover of $X$ has a refinement in which no point of $X$ belongs to more than $(k+1)$ elements.
In this paper, we say a topological space $X$ is {\it $k$-dimensional} if the Lebesgue covering dimension of $X$ is $k$.
We say a space is {\it finite dimensional} if the Lebesgue covering dimension of the space is finite.

We say that an isometry $\gamma$ of a metric space $(X,d)$ is {\it semi-simple} 
if there exists $x\in X$ such that $d(x, \gamma(x))=\inf\{d(y, \gamma(y))\mid y\in X\}$.
An isometric group action is called {\it semi-simple} if every group element acts as a semi-simple isometry.

In the following, we assume that all group actions are isometric, and that all metric spaces are complete.

\begin{theorem}[{\cite[Proposition 3.4]{Bridson}}]\label{Bridson}
Let $X$ be a $k$-dimensional CAT(0) space. 
Let $k_1,\ldots, k_l$ be positive integers such that $0 < k < \Sigma_{i=1}^{l}k_i$. 
Let $S_1,\ldots, S_l$ be finite subsets of $\Isom(X)$.
For all $i\neq j \in \{1, . . . , l\}$, suppose that 
\begin{itemize}
\item $s_i$ and $s_j$ commute for all $s_i\in S_i$ and $s_j\in S_j$, and that
\item $S_i$ is a conjugate of $S_j$ in $\Isom(X)$.
\end{itemize}
If each $k_i$-element subset of $S_i$ has a fixed point in $X$ for every $i\in \{1, . . . , l\}$, 
then for every $i\in \{1, . . . , l\}$, elements of $S_i$ have a common fixed point. 
\end{theorem}
Theorem~\ref{Bridson} is a variation of Helly's Theorem for convex subsets in Euclidean spaces.

\begin{theorem}[{\cite[Theorem 1.1]{Kato2}}]\label{MainNew}
Let $k\in \mathbb{N}$. Let $G$ be a group acting faithfully on a set $A$. 
Let $g\in G$ be an element satisfying the following conditions:
there exists a sequence of subgroups $\langle g\rangle = H_0<H_1<\cdots<H_k<H_{k+1}=G$ 
with elements $g_i\in H_{i+1}$ $(1\leq i\leq k)$ such that for every $1\leq i\leq k$,
 \begin{itemize}
 \item every homomorphism from $H_i$ to $\mathbb{R}$ is trivial, and
 \item $g_i(\supp(H_i))\cap \supp(H_i)=\emptyset$ in $A$.
 \end{itemize}  
Then for every semi-simple action of $G$ on a $k$-dimensional CAT(0) space, $g$ has a fixed point. 
\end{theorem}

In the following, we argue fixed point properties of ring groups. 
We start with a lemma for Thompson's group $F$.
The proof can be found in the proof of \cite[Corollary 3.3]{Kato2}, but we summarize the proof for the readers' convenience.
\begin{lemma}\label{F}
For every $f\in [F,F]$
and every semi-simple action of $F$ on a finite-dimensional CAT(0) space, $f$ has a fixed point.
\end{lemma}
\begin{proof}
We fix $k\in \mathbb{N}$ and a semi-simple action of $F$ on a $k$-dimensional CAT(0) space $X$ arbitrarily.
Let $\{I_i\}_{1\leq i\leq k}$ be a sequence of standard dyadic intervals in $(0,1)$
such that each $I_i$ is contained in the interior of $I_{i+1}$.

Without loss of generality, we may assume that $\supp(f)$ is contained in the interior of $I_1$.
Indeed, if not, we take a closed interval $I\supset \supp(f)$ (see Lemmas~\ref{4.4} (5) and \ref{F_n-chain_lem} $(1)$),
and an element $g\in F$ that maps $I$ in the interior of  $I_1$ (see Lemma~\ref{lemma-for-fact} $(2)$).
Then $\supp({g}fg^{-1})=g\supp(f)$ is contained in $I_1$.
When ${g}fg^{-1}$ fixes a point $x\in X$, $f$ also fixes a point $g(x)\in X$.

For every $1\leq i\leq k$, let $H_i$ be the commutator subgroup of $F(I_i)=\{h\in F\mid \supp(h)\subset I_i\}<F$, and let $H_{k+1}=F$.
Through the canonical isomorphism $\psi_{n,I_i}:F(I_i)\to F$, every $H_i$ is isomorphic to the commutator subgroup of $F$ and thus nonabelian and simple.
It follows that $\{H_i\}_{1\leq i\leq k}$ satisfies the first assumption in Theorem~\ref{MainNew}.
We take elements $\{g_i\}_{1\leq i\leq k}$ such that $g_i\in H_{i+1}$ and $g_i(I_i)\cap I_i=\emptyset$ (see Lemmas~\ref{facts2} and \ref{4.4} (3)).
By construction,  $\{H_i\}_{1\leq i\leq k+1}$ and $\{g_i\}_{1\leq i\leq k}$ satisfy the second assumption in Theorem~\ref{MainNew}.
Therefore, according to Theorem~\ref{MainNew}, $f$ fixes a point in $X$.
\end{proof}

\begin{theorem}\label{FnFAk}
Let $m\geq 2$.
Let $G_{\mathcal{F}}$ be an $m$-chain group with respect to the generating set $\mathcal{F}=\{f_i\}_{1\leq i\leq m}$.
Let $H$ be a finitely generated subgroup of $[G_{\mathcal{F}}, G_{\mathcal{F}}]$.
Then for every semi-simple action of $G_{\mathcal{F}}$ on a finite-dimensional CAT(0) space, elements of $H$ have a common fixed point.
\end{theorem}

\begin{proof}
We fix $k\in \mathbb{N}$ and a semi-simple action of $G_{\mathcal{F}}$ on a $k$-dimensional CAT(0) space $X$ arbitrarily.
Let 
$$H_i:=\langle f_i, f_{i+1}\rangle \quad(1\leq i\leq m-1).$$

First, we construct a finite subset $S'$ of $G_{\mathcal{F}}$ such that 
\begin{itemize}
\item $H<\langle S'\rangle$, and
\item for every $s\in S'$, there exists $i\in \{1,\ldots, m-1\}$ such that $\supp(s')$ is contained in a closed interval in $\supp(H_{i})$.
\end{itemize}
Let 
$$c_i:=[f_i, f_{i+1}]\quad (1\leq i\leq m-1).$$
Since $[G_{\mathcal{F}}, G_{\mathcal{F}}]$ is generated by conjugates of $\{c_i\}_{1\leq i\leq m-1}$,
there exists a finite generating set $S=\{c'_{j}\}_{1\leq j\leq n}$ of $H$ consisting of conjugates of $c_i$.
By Lemma~\ref{4.4} (5), every $\supp(c_i)$ is contained in a closed interval in $\supp(H_i)$.
Since each $c'_j$ is a conjugate of some $c_i$, $\supp(c'_j)$ is contained in a closed interval $[a_j, b_j]$ in $(0,1)$.
By assumtions on $\mathcal{F}$, there exists $N_j\in \mathbb{N}$ such that for every $i\in \{1,\ldots,m-2\}$, 
$$f_{i}^{N_j}\cdots f_2^{N_j}f_1^{N_j}(a_j)>\partial_{-}\supp(f_{i+1}).$$
Applying Lemma~\ref{4.4} (1) for each chain subgroup $H_i$, there exists $g_{i,j}\in [H_i,H_i]$ such that 
$$g_{i,j}\left(f_{i-1}^{N_j}\cdots f_2^{N_j}f_1^{N_j}(a_j)\right)=f_i^{N_j}\left(f_{i-1}^{N_j}\cdots f_2^{N_j}f_1^{N_j}(a_j)\right).$$ 
Again by Lemma~\ref{4.4} (5), every $\supp(g_{i,j})$ is contained in a closed interval in $\supp(H_i)$.
Let 
$$g_j:=g_{j,m-2}\cdots g_{j,1}.$$
Since $g_j$ is order preserving, 
$$\partial_{-}\supp(f_{m-1})<g_j(a_j)<g_j(b_j)<\partial_{+}\supp(f_{m}).$$
It follows that
$$\supp({g_{j}}c'_j{g_{j}}^{-1})=g_j(\supp(c'_j))\subset [g_j(a_j), g_j(b_j)] \subset \supp(H_{m-1}).$$
Let 
\begin{align}
S':=\{{g_{j}}c'_j{g_{j}}^{-1}\}_{1\leq j\leq l}\cup \{g_{j,i}\}_{1\leq j\leq l, 1\leq i\leq m-2}.
\end{align} 
This $S'$ satisfies the requirements.

Next, we show that every element of $S'$ has a fixed point in $X$.
Applying Lemma~\ref{F} to the induced action of $H_i\cong F$ on $X$, we see that every element of $[H_i,H_i]$ fixes a point in $X$.
Therefore, all $c_i$ and $g'_{j,i}$ has a fixed point.
It follows that every ${g_{j}}c'_j{g_{j}}^{-1}$ has a fixed point, since it is a conjugate of some $c_i$.

Finally, we show that elements of $S'$ have a common fixed point in $X$.
By the second requirement for $S'$, the support of every $(k+1)$-element subset $S'_k$ of $S'$ is contained in a closed interval in $\supp(G_{\mathcal{F}})$.
By Lemma~\ref{4.4} (3), we can take $(k+1)$ elements $h_1,\ldots, h_{k+1}$ of $G_{\mathcal{F}}$ which map $\supp(S'_{k})$ into $(k+1)$ disjoint open subsets of $\supp(G_{\mathcal{F}})$. 
We apply Theorem~\ref{Bridson} by setting $k_1=\cdots=k_l=1$ and $S_i$ to be the conjugate of $S'_k$ by $h_i$ for every $i\in \{1,\ldots k+1\}$. 
It follows that elements of $S'$ have a common fixed point.
\end{proof}

\begin{corollary}\label{fixed_Fn}
Let $n\geq 2$. 
Let $H$ be a finitely generated subgroup of $[F_n, F_n]$.
Then for every semi-simple action of $F_n$ on a finite-dimensional CAT(0) space, elements of $H$ have a common fixed point.
\end{corollary}

\begin{remark}\label{chain_nFA}
For any semi-simple action of a chain group $G_{\mathcal{F}}$ on any finite-dimensional CAT(0) space,
$G_{\mathcal{F}}$ does not necessarily have a global fixed point.
In fact, $G_{\mathcal{F}}$ does not admit Serre's property FA.
Property FA for a group $G$ is equivalent to the three conditions: $G$ is finitely generated, 
$G$ does not surject on $\mathbb{Z}$, and $G$ is not an amalgamated product (\cite{Serre}).
We fix a map $w:G_{\mathcal{F}}\to \cup_{l\in \mathbb{N}}{(\mathcal{F}\cup \mathcal{F}^{-1})}^{l}$,
which maps each element $g\in G_{\mathcal{F}}$ to a finite word in $\mathcal{F}\cup \mathcal{F}^{-1}=\{f_i^{\pm}\}_{1\leq i\leq m}$ which represents $g$. 
Let 
$$p_1(g)= (\text{the number of } f_1 \text{ in } w(g)) - (\text{the number of } f_1^{-1} \text{ in } w(g)).$$
By observing the action of $g$ on a neighborhood of $0\in [0,1]$,
we can check that $p_1(g)$ does not depend on the choice of $w$.
It follows that $p_1:G_{\mathcal{F}}\to \mathbb{Z}$ is a well-defined surjection, and thus $G_{\mathcal{F}}$ does not have property FA.
\end{remark}


\begin{lemma}\label{minimal_lem}
Let $m$, $m'$, $G_{\mathcal{F}}$, and $H$ be those in the statement of Theorem~\ref{TnFAk}.
Let 
$$\mathcal{F}=\{f_i\}_{1\leq i\leq m},$$
and
$$\mathcal{F}':=\{f_i\}_{1\leq i\leq m'}.$$
Without loss of generality, we assume that the subgroup generated by $\mathcal{F}'$, which we write $G_{\mathcal{F}'}$, is minimal.
Let 
$$H_i:=\langle f_i, f_{i+1}\rangle\quad(1\leq i\leq m),$$
where indices are considered modulo $m$.

Let $I$ be a closed interval in $S^1$ such that the closure of $S^1\setminus I$ is contained in the support of a chain subgroup $G$ of $G_{\mathcal{F}}$.
Then there exists a sequence of elements 
$$g_{1},\ldots, g_{l} \in  [G_{\mathcal{F}'}, G_{\mathcal{F}'}] \cup \bigcup_{m'\leq i\leq m}[H_i,H_i]$$ such that
$$({g_{l}\cdots g_{1}})(I)\subset \supp(H_{m}).$$
\end{lemma}

\begin{proof}
As in the proof of Theorem~\ref{FnFAk}, we apply Lemma~\ref{4.4} (2) repeatedly for chain subgroups of $G$, and get $h_{1},\ldots, h_{l'}\in  [G_{\mathcal{F}'}, G_{\mathcal{F}'}] \cup \bigcup_{m'\leq i\leq m}[H_i,H_i]$ such that 
$$(h_{l'}\cdots h_{1}) \left(cl(S^1\setminus I)\right)\subset\supp(G_{\mathcal{F}'}),$$
where $cl(S^1\setminus I)$ is the closure of $S^1\setminus I$.
By the minimality of $G_{\mathcal{F}'}$, we may apply Lemma~\ref{4.4} (3): 
for any $\varepsilon>0$, there exists $h\in [G_{\mathcal{F}'},G_{\mathcal{F}'}]$ such that 
$${h}^{-1} ([\partial_{-}\supp(f_1)+\varepsilon, \partial_{+}\supp(f_{m'})-\varepsilon]) \subset (h_{l'}\cdots h_{1})(S^1\setminus  I_j). $$
Therefore, 
$$ (hh_{l'}\cdots h_{1})(I)\subset S^1\setminus [\partial_{-}\supp(f_1)+\varepsilon, \partial_{+}\supp(f_{m'})-\varepsilon].$$
For sufficiently small $\varepsilon>0$, 
$$S^1\setminus [\partial_{-}\supp(f_1)+\varepsilon, \partial_{+}\supp(f_{m'})-\varepsilon]\subset \supp\left(\langle f_{m'+1},\ldots, f_m, f_1\rangle\right),$$
and thus 
$$(hh_{l'}\cdots h_{1})(I)\subset \supp\left(\langle f_{m'+1},\ldots, f_m\rangle\right).$$
Again, we apply Lemma~\ref{4.4} (2) repeatedly for chain subgroups of $\langle f_{m'+1},\ldots, f_m\rangle$ as in the proof of Theorem~\ref{FnFAk}, and get a sequence of elements $h'_1,\ldots, h'_{l''}$ of $\bigcup_{m'+1\leq i\leq m}[H_i,H_i]$ such that
$$(h'_{l''}\cdots h'_1hh_{l'}\cdots h_{1})(I)\subset\supp(H_m).$$
Therefore, 
$$(g_{1},\ldots, g_{l}):=(h_1,\ldots,h_{l'},h,h'_1,\ldots,h'_{l''})$$
satisfies the requirements.
We show an example of the construction of the sequence of elements in Figure~\ref{ringfigure}.
\begin{figure}
\begin{center}
\scalebox{1}{
\includegraphics[width=10cm,pagebox=cropbox,clip]{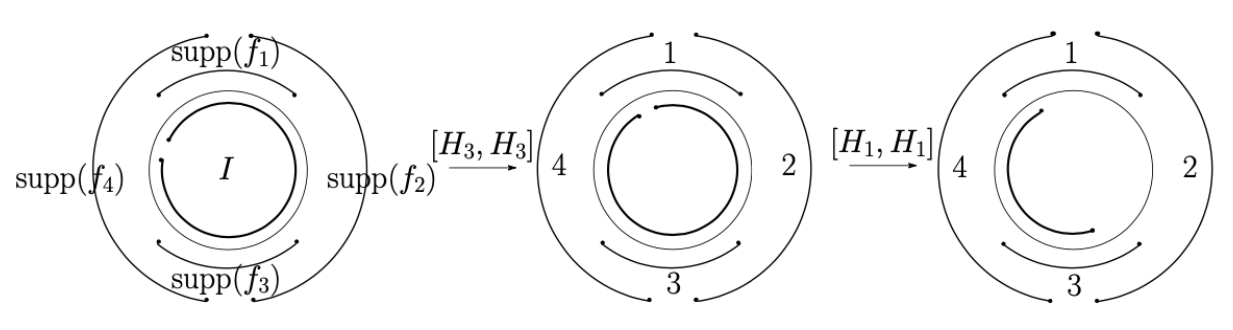}
}
\caption{A construction of the sequence of elements whose composition maps $I$ in $\supp(H_{m})$, where $m=4$ and $m'=2$.}
\label{ringfigure}
\end{center}
\end{figure}
\end{proof}

\begin{proof}[Proof of Theorem~\ref{TnFAk}]
Let $m$, $m'$, $G_{\mathcal{F}}$, and $H$ be those in the statement of Theorem~\ref{TnFAk}.
We fix $k\in \mathbb{N}$ and a semi-simple action of $G_{\mathcal{F}}$ on a $k$-dimensional CAT(0) space $X$ arbitrarily.
Let 
$$\mathcal{F}=\{f_i\}_{1\leq i\leq m},$$
and
$$\mathcal{F}':=\{f_i\}_{1\leq i\leq m'}.$$
Without loss of generality, we assume that the subgroup generated by $\mathcal{F}'$, which we write $G_{\mathcal{F}'}$, is minimal.
According to Lemma~\ref{5.7}, we may exchange generators $f_{m'+1},\ldots, f_{m}$ and assume that $m>>k, m'$.
Let 
$$H_i:=\langle f_i, f_{i+1}\rangle\quad(1\leq i\leq m),$$
where indices are considered modulo $m$.

First, we construct a finite subset $S'$ of $G_{\mathcal{F}}$ such that 
\begin{itemize}
\item $H<\langle S'\rangle$, and
\item for every $s'\in S'$, $\supp(s')$ is contained in a closed interval in $\supp(G_{\mathcal{F}'})$, or there exists $i\in \{m',\ldots, m\}$ such that $\supp(s')$ is contained in a closed interval in $\supp(H_{i})$.
\end{itemize}
Let 
$$c_i:=[f_i, f_{i+1}]\quad (1\leq i\leq m),$$
where indices are considered modulo $m$.
Since $[G_{\mathcal{F}}, G_{\mathcal{F}}]$ is generated by conjugates of $\{c_i\}_{1\leq i\leq m}$,
there exists a finite generating set $S=\{c'_{j}\}_{1\leq j\leq n}$ of $H$ consisting of conjugates of $c_i$.
We claim that, for every $j$, there exists a sequence of elements 
$$g_{j,1},\ldots, g_{j,l_j} \in  [G_{\mathcal{F}'}, G_{\mathcal{F}'}] \cup \bigcup_{m'\leq i\leq m}[H_i,H_i]$$ such that
$$\supp\left(({g_{j,l_j}\cdots g_{j,1}})c'_j{({g_{j,l_j}\cdots g_{j,1}})}^{-1}\right)\subset \supp(H_{i_j})$$
for some $i_j\in \{1,\ldots, m\}$.
Since $c'_j\in [G_{\mathcal{F}}, G_{\mathcal{F}}]$, 
there exists a closed interval $I_j$ in $S^1$ such that 
$$\supp(c'_j)\subset I_j$$  for every $j$ (Lemma~\ref{4.4} (5)).
Note that either $I_j$ or the closure of $S^1\setminus I_j$ is contained in the support of a chain subgroup of $G_{\mathcal{F}}$, which we name $G_{j}$.
If $S^1\setminus I_j$ is in the support of a chain subgroup, the claim follows from Lemma~\ref{minimal_lem}. 
We note that this is the only point where we use the minimality of the chain subgroup $G_{\mathcal{F}'}$.
Suppose that $I_j$ is in the support of a chain subgroup.
As in the proof of Theorem~\ref{FnFAk}, we apply Lemma~\ref{4.4} (2) repeatedly for chain subgroups of $G_{j}$, and get $g_{j,1},\ldots, g_{j,l_j}\in  [G_{\mathcal{F}'}, G_{\mathcal{F}'}] \cup \bigcup_{m'\leq i\leq m}[H_i,H_i]$ such that 
$$g_{j,1}\cdots g_{j,l_j} (I_j)\subset \supp(H_{i_j})$$ for some $i_j$.
Therefore, 
\begin{align*}
&\supp\left(({g_{j,l_j}\cdots g_{j,1}})c'_j{({g_{j,l_j}\cdots g_{j,1}})}^{-1}\right)=({g_{j,l_j}\cdots g_{j,1}})\left(\supp\left(c'_j\right)\right)\\
&\subset ({g_{j,l_j}\cdots g_{j,1}})(I_j)\subset \supp(H_{i_j})
\end{align*}
for some $i_j$.
Hence the claim follows.
Let 
\begin{align}
S'=\{({g_{j,l_j}\cdots g_{j,1}})c'_j{({g_{j,l_j}\cdots g_{j,1}})}^{-1}\}_{1\leq j\leq n}\cup \{g_{j,i_j}\}_{1\leq i_j\leq l_j, 1\leq j\leq n}.
\end{align} 
This $S'$ satisfies the requirements.

Next, we show that every element of $S'$ has a fixed point in $X$.
Applying Lemma~\ref{F} to the induced action of $H_i\cong F$ on $X$, we see that every element of $[H_i,H_i]$ has a fixed point.
It follows that their conjugates $^{g_{j,l_j}\cdots g_{j,1}}c'_j$ also have fixed points.
Applying Theorem~\ref{FnFAk} to the induced action of $G_{\mathcal{F}'}$ on $X$, we see that every element of $[G_{\mathcal{F}'},G_{\mathcal{F}'}]$ has a fixed point.

Finally, we show that elements of $S'$ have a common fixed point in $X$.
Since $m>>k, m'$, the support of every $(k+1)$-element subset $S'_k$ of $S'$ is contained in a closed interval in the support of an $(m-1)$-chain subgroup $L$ of $G_{\mathcal{F}}$.
By applying Lemma~\ref{4.4} (3) to $L$, we can take $(k+1)$ elements of $G_{\mathcal{F}}$ which map $\supp(S'_{k+1})$ into $(k+1)$ disjoint open subsets of $\supp(L)$. 
By applying Theorem~\ref{Bridson} to $S'$, it follows that elements of $S'$ have a common fixed point.

\end{proof}

\begin{proof}[Proof of Corollary~\ref{Tn_cor}]
By Proposition~\ref{Tn_ring}, each $T_n$ is an $(n^2+n)$-ring group with a minimal $n$-chain subgroup. 
For every $n\geq 3$, $T_n$ is simple (\cite[Theorem 4.15]{Brown}). 
In particular, $[T_n,T_n]=T_n$ is finitely generated.
Then we can apply Theorem~\ref{TnFAk} for $G_{\mathcal{F}}=T_n$ and $H=[T_n,T_n]=T_n$.
\end{proof}





\end{document}